\documentclass[11pt, a4paper]{article}

\usepackage{amsmath, amssymb, amsthm, xspace, a4wide, amscd}
\usepackage[english]{babel}
\usepackage[T1]{fontenc}
\usepackage[latin1]{inputenc}

\newtheorem{theorem}{Theorem}[section]
\newtheorem{lemma}{Lemma}[section]
\newtheorem{corollary}{Corollary}[section]
\newtheorem{assumption}{Assumption}

\theoremstyle{definition}
\newtheorem{definition}{Definition}[section]
\newtheorem{example}{Example}[section]

\theoremstyle{definition}

\begin{document}

\begin{center}\Large
Stratonovich-type integral with respect to a general stochastic measure \footnote{To appear in {\em Stochastics: An International Journal of Probability and Stochastic Processes}}
\end{center}

\begin{center}
Vadym Radchenko \footnote{Department of Mathematical Analysis, Taras Shevchenko National University of Kyiv, Kyiv 01601, Ukraine.
\emph{E-mail adddress}: {vradchenko@univ.kiev.ua}}
\end{center}

%\emph{2010 Mathematics Subject Classication}: 60H05; 60G57; 60H10; 60G17

%\emph{Keywords}: stochastic measure; Stratonovich integral; Stochastic integral; Strong cubic variation; Doss--Sussmann transformation

\begin{abstract}
Let $\mu$ be a general stochastic measure, where we assume for $\mu$ only $\sigma$-additivity in probability and continuity of paths. We prove that the symmetric integral $\int_{[0,T]}f(\mu_t, t)\circ\,{\rm d}\mu_t$ is well defined.
For stochastic equations with this integral, we obtain the existence and uniqueness of a solution.
\end{abstract}

\section{Introduction}
\label{scintr}

The main object of this paper is to construct an integral with respect to a general stochastic measure. We will define the integral of the form $\int_{[0,T]}f(\mu_t, t)\circ\,{\rm d}\mu_t$, where $\mu_t=\mu((0,t])$, $\mu$ is the stochastic set function defined on the Borel $\sigma$-algebra of $[0,T]$, and for $\mu$ we assume $\sigma$-additivity in probability and continuity of the paths of $\mu_t$. No additional assumptions will be made for $\mu$ regarding moment existence, path regularity or martingality condition.

We will define the ``symmetric'' integral as the limit in probability of Stratonovich integral sums. This integral is well defined for $f\in {\mathbb C}^{1,1}({\mathbb R}^2)$, and the chain rule formula holds in this case. In the last Section of the paper we will show that the limit of ``non-anticipated'' integral sums does not exist for some $\mu$. Therefore, the definition of It\^{o} type integral may not be used for general stochastic measures. But in some cases, such as stochastic integration on manifolds or Wong-Zakai approximation of stochastic equations, Stratonovich integral is a very useful tool. Stratonovich-type definitions of stochastic integrals with respect to differrent classes of processes may by founded in~\cite{kurtz}, \cite{deya13}, \cite{hu2013}, and \cite{swanson}.

Integrals of deterministic functions with respect to general stochastic measures are well studied, see \cite{kwawoy}, \cite{curbera}, \cite{radmon}.  Stochastic partial differential equations with this integral were considered in~\cite{rads09}, \cite{radc14}.

In order to study a wider class of equations driven by stochastic measures, we need a definition of an integral for random functions. This will be done in this paper. Our approach is similar to~\cite{erru03}. Moreover, under some additional assumptions, we prove that $\mu_t$ is the process with finite strong cubic variation. Further, we define the integral without these additional assumptions.

The paper is organized as follows. In Section~\ref{scprel} we have compiled some basic facts about stochastic measures.
In Section~\ref{scsmcv} we study cubic variation of $\mu_t$. Section~\ref{scstri} is devoted to the definition of our integral. Section~\ref{scsdes} studies existence and uniqueness of a solution to equations driven by stochastic measures. To this aim we apply the Doss-Sussmann transformation. In Section~\ref{sccoun} we give two important counterexamples, where we demonstrate that the finite quadratic variation approach and the ``non-anticipated'' definition of integral can not be applied in our case.

\section{Preliminaries}
\label{scprel}

Let ${\sf L}_0={\sf L}_0(\Omega, {\mathcal F}, {\sf P} )$ be the set of all real-valued
random variables defined on the complete probability space $(\Omega, {\mathcal F}, {\sf P} )$ (more precisely, the set of equivalence classes). Convergence in ${\sf L}_0$ means the convergence in probability. Let ${\sf X}$ be an arbitrary set and ${\mathcal{B}}$ a $\sigma$-algebra of subsets of ${\sf X}$.

\begin{definition}
A $\sigma$-additive mapping $\mu:\ {\mathcal{B}}\to {\sf L}_0$ is called {\em stochastic measure} (SM).
\end{definition}

In other words, $\mu$ is a vector measure with values in ${\sf L}_0$. We do not assume positivity or moment existence
for SM. In~\cite{kwawoy} such $\mu$ is called a general SM. In the sequel, $\mu$ denotes a SM.

For a deterministic measurable function $g:{\sf X}\to{\mathbb R}$ and a SM~$\mu$, an integral of the form $\int_{\sf
X}g\,{\rm d}\mu$ is defined and studied in~\cite[Chapter 7]{kwawoy}, see also~\cite{curbera}, \cite{radmon}. In particular, every bounded measurable $g$ is integrable with respect to any~$\mu$. An analogue of the Lebesgue dominated convergence theorem holds for this integral, see Proposition 7.1.1 of~\cite{kwawoy}.

The following analogue of Nikodym theorem is satisfied for SMs.

\begin{theorem}\cite[Theorem 8.6]{dretop}\label{thniko}
Let $\mu_n$ are SMs on ${\mathcal{B}}$, $n\ge 1$, and
\[
\forall\, {\sf A}\in{\mathcal{B}}\ \exists\ \mu({\sf A})={\rm p}\lim_{n\to\infty} \mu_n ({\sf A}).
\]
Then $\mu$ is a SM on ${\mathcal{B}}$.
\end{theorem}

Main results of this paper will be proved for $\mu$ defined on the $\sigma$-algebra of Borel subsets of $[0, T]$ under the assumption that the random process $\mu_{t}:=\mu((0,t])$, $0\le t\le T$ has continuous paths.

Examples of SMs on $[0,T]$ are the following:

\begin{enumerate}

\item Let $M_{t}$ be a square integrable martingale. Then $\mu({\sf A})=\int_0^T {\bf 1}_{\sf
A}(t)\,{\rm d}M_{t}$ is a SM.

\item If $W^H_{t}$ is a fractional Brownian motion with Hurst index~$H>1/2$ and $f: [0,
T]\to\mathbb{R}$ is a bounded measurable function then $\mu({\sf A})=\int_0^T f(t){\bf 1}_{\sf A}(t)\,{\rm d}W^H_{t}$ is a SM. This follows from~\cite[Theorem~1.1]{memiva}.

\item An $\alpha$-stable random measure defined on a $\sigma$-algebra is
an SM, see \cite[Chapter 3]{samtaq}.

\item\label{exserm} Let the random series $\sum_{n\ge 1}\xi_n$ converge unconditionally in probability, and $m_n$ are real signed measures on~${\mathcal{B}}$, $|{\sf m}_n({\sf A})|\le 1$. Set $\mu({\sf A})=\sum_{n\ge 1}\xi_n {\sf m}_n({\sf A})$. Convergence of this series in probability follows from~\cite[Theorem~V.4.2]{vahtar}, and $\mu$ is a SM by Theorem~\ref{thniko}.

\end{enumerate}

Theorem~8.3.1 of~\cite{kwawoy} states the conditions under which the increments
of a real-valued process with independent increments generate a SM.

\section{SM has finite strong cubic variation}
\label{scsmcv}

Following \cite{rusval00}, \cite{erru03} and taking into account Remark~2.1 from~\cite{covrus},
we say that a continuous process $X_{t}$, $0\le t\le T$, has the {\em strong $n$-variation} on $[0,T_1]$ ($T_1<T$) if
\[
[X;n](t)=\lim_{\varepsilon\to 0+}\frac{1}{\varepsilon}\int_{[0,t]} (X_{s+\varepsilon}-X_{s})^n\,{\rm d}s,\quad 0\le t\le T_1,
\]
exists in the sense of uniform convergence in probability, and every sequence $\varepsilon_i\to 0+$ admits subsequence $\tilde{\varepsilon}_{i}$ such that
\[
\sup_i\,\frac{1}{\tilde{\varepsilon}_{i}}\int_{[0,T_1]} |X_{s+\tilde{\varepsilon}_{i}}-X_{s}|^n\,{\rm d}s<+\infty\quad {\rm a.~s.}
\]

\begin{assumption}\label{ass1}
There exists a real-valued  {f}{i}nite  measure~${\sf m}$ on~$\left({\sf X}, {\mathcal{B}}\right)$ with the following property: if a measurable function $h:{\sf X}\to\mathbb{R}$ is such that \mbox{$\int_{\sf X}h^2\,{\rm d}{\sf m}<+\infty$} then $h$ is integrable with respect to~$\mu$ on~${\sf X}$.
\end{assumption}

This assumption holds, for example for SM generated by $B^H_{t}$ with $H>1/2$, for $\alpha$-stable random measures on~${\mathcal B}$ with $\alpha\in(0,2]$ and for orthogonal SMs.

\begin{lemma}\cite[Lemma 3.3]{radt09}\label{lmfmus} Suppose Assumption~\ref{ass1} holds. Let a measurable functions $f_k:{\sf X}\to{\mathbb R}$, $k\ge 1,$ satisfy
\[
\int_{\sf X}\Bigl( \sum_{k=1}^{\infty} f_k^2\Bigr){\rm d}{\sf m}<+\infty.
\]
Then
\[
\sum_{k=1}^{\infty} \Bigl(\int_{\sf X} f_k\,{\rm d}\mu\Bigr)^2< +\infty\quad\mbox{\textrm ~a.~s.}
\]
\end{lemma}

\begin{corollary}\label{crsumf}
If Assumption~\ref{ass1} holds then the set of random variables
\[
\Bigl\{\sum_{k=1}^{j} \Bigl(\int_{\sf X} f_k\,{\rm d}\mu\Bigr)^2\quad \Bigr|\quad f_k:{\sf X}\to{\mathbb R}\ {\rm are\ measurable},\quad \sum_{k=1}^{j} f_k^2(x)\le 1,\quad j\ge 1\Bigr\}
\]
is bounded in probability.
\end{corollary}

\begin{proof} If the statement fails, for some $\alpha>0$ and all $n\ge 1$ we can find functions $f_{kn}$, $1\le k\le j_n$, such that
\[
\sum_{k=1}^{j_n} f_{kn}^2(x)\le 1,\quad {\sf P}\Bigl\{\sum_{k=1}^{j_n} \Bigl(\int_{\sf X} f_{kn}\,{\rm d}\mu\Bigr)^2>2^n\Bigr\}>\alpha.
\]
Then
\[
\sum_{n=1}^{\infty}\sum_{k=1}^{j_n} \bigl(2^{-n/2}f_{kn}(x)\bigr)^2\le 1,\quad \sum_{n=1}^{\infty}\sum_{k=1}^{j_n} \Bigl(\int_{\sf X} 2^{-n/2}f_{kn}\,{\rm d}\mu\Bigr)^2\ {\rm diverges},
\]
that contradicts Lemma~\ref{lmfmus}. \end{proof}

\begin{theorem}\label{thficv}
Let $\mu$ be a SM on Borel subsets of $[0,T]$, Assumption~\ref{ass1} holds, the process $\mu_{t}=\mu((0,t])$, $0\le t\le T$ has Riemann integrable paths, $0<T_1<T$.

Then the set of random variables
\begin{equation}\label{eqinmu}
\int_{[0,T_1]}\frac{\bigl|\mu_{s+\varepsilon}-\mu_{s}\bigr|^{2}}{\varepsilon}\,{\rm d}s,\quad 0<\varepsilon<T-T_1,
\end{equation}
is bounded in probability.
\end{theorem}

\begin{proof} For any $n\ge 1$ take the partition of $[0,T_1]$ by points $s_{kn}=\frac{k}{n}T_1\varepsilon\wedge T_1$, $0\le k\le j_n$, and consider Riemann integral sum for~\eqref{eqinmu}
\begin{multline}
\sum_{k=1}^{j_n} \frac{\bigl|\mu_{s_{kn}+\varepsilon}-\mu_{s_{kn}}\bigr|^{2}}{\varepsilon}\,\frac{T_1\varepsilon}{n}
 =  T_1\sum_{k=1}^{j_n} \frac{\bigl|\mu((s_{kn},s_{kn}+\varepsilon])\bigr|^{2}}{n}
=T_1\sum_{k=1}^{j_n} \Bigl(\int_{[0,T_1]} f_{kn}\,{\rm d}\mu\Bigr)^2,\\  {\rm where}  \quad f_{kn}(x)=\frac{1}{\sqrt{n}}\,{\bf 1}_{(s_{kn},s_{kn}+\varepsilon]}(x).\quad \label{eqriems}
\end{multline}
(we take the length of the last segment of the partition equal to $T_1\varepsilon/n$ and the sum does not decrease).
We have $\sum_{k=1}^{j_n} f_{kn}^2(x)\le 1$, by Corollary~\ref{crsumf} set of sums~\eqref{eqriems} is bounded in probability, which implies the statement of the theorem.
\end{proof}

Now we easily obtain the following result for continuous processes $\mu_{t}$.

\begin{corollary}
Let $\mu$ be a SM on Borel subsets of $[0,T]$, Assumption~\ref{ass1} holds and the process $\mu_{t}=\mu((0,t])$, $0\le t\le T$ has continuous paths.
Then for any $\delta>0$ and $T_1$, $0<T_1<T$ we have
\[
\frac{1}{\varepsilon}\int_{[0,T_1]}\bigl|\mu_{s+\varepsilon}-\mu_{s}\bigr|^{2+\delta}\,{\rm d}s\stackrel{\sf P}{\to} 0,\quad \varepsilon\to 0+.
\]
In particular, the process $\mu_{t}$ has the strong cubic variation $[\mu;3](t)=0$, $0\le t\le T_1$.
\end{corollary}

Therefore, under Assumption~\ref{ass1}, for $\mu$ with continuous paths the stochastic calculus developed in~\cite{erru03} may be applied. For instance, if $F\in {\mathbb C}^{3}({\mathbb R})$ then we have that
\[
F(\mu_{t})=F(\mu_{0})+\int_{[0,t]}F'(\mu_{s})\,{\rm d}^{\circ\, }\mu_{s},\quad 0<t<T,
\]
and this symmetric integral is well defined in the sense of~\cite{erru03}, see Proposition~3.3.

In the following section we will give a Stratonovich-type definition of the integral with respect to $\mu$ without assuming Assumption~\ref{ass1}.

\section{Stratonovich-type integral}
\label{scstri}

\begin{lemma} \cite[Lemma]{rads06}
\label{lmfkmu}
Let $\mu$ be a SM and $a_n,\ n\ge 1$, a sequence of positive numbers such that
$\sum_{n=1}^{\infty} a_n < \infty$. Let $\Delta_{kn}\in \mathcal{B},\ n\ge 1,\ 1\le k\le j_n$, be such that for each
$n$ and $k_1\ne k_2\ \Delta_{k_1 n}\cap \Delta_{k_2 n}=\emptyset$. Then
\[
\sum_{n=1}^{\infty} {a_n^2}\sum_{k=1}^{j_n} \mu^2\bigl(\Delta_{kn}\bigr) <\infty\quad \mbox{a.~s.}
\]
\end{lemma}

\begin{corollary}\label{crmdkn}
The set of random variables
\[
\Bigl\{\sum_{k=1}^{j} \mu^2\left(\Delta_{kn}\right)\quad \Bigr|\quad \Delta_{kn}\in \mathcal{B},\quad \Delta_{k_1 n}\cap \Delta_{k_2 n}=\emptyset\ {\rm for}\ k_1\ne k_2,\quad j\ge 1\Bigr\}
\]
is bounded in probability.
\end{corollary}

\begin{proof}
If the statement fails, for some $\alpha>0$ for all $n\ge 1$ we can find disjoint sets $\Delta_{kn}\in \mathcal{B}$, $1\le k\le j_n$, such that
\[
{\sf P}\Bigl\{\sum_{k=1}^{j_n} \mu^2\bigl(\Delta_{kn}\bigr)>2^n\Bigr\}>\alpha.
\]
Then for these sets and $a_n=2^{-n/2}$ we have a contradiction with Lemma~\ref{lmfkmu}.
\end{proof}

\begin{definition} Let $\xi_{t}$ and $\eta_{t}$ be random processes on $[0,T]$, $0=t_{0}^n<t_{1}^n<\dots<t_{j_n}^n=T$ be a sequence of partitions such that $\max_k |t_k^n-t_{k-1}^n|\to 0$, $n\to\infty$. We define
\begin{equation}\label{eqdfis}
\int_{(0,T]}\xi_{t}\circ\,{\rm d}\eta_{t}:={\rm p}\lim_{n\to\infty}\sum_{k=1}^{j_n}\frac{\xi_{t_{k-1}^n}+\xi_{t_{k}^n}}{2}\,\bigl(\eta_{t_{k}^n}- \eta_{t_{k-1}^n}\bigr)
\end{equation}
provided that this limit in probability exists.
\end{definition}

\begin{assumption}\label{ass2}
$\mu$ is a SM on Borel subsets of $[0,T]$, and the process $\mu_{t}=\mu((0,t])$ has continuous paths on $[0,T]$.
\end{assumption}

\begin{assumption}\label{ass3}
$V_{t}$ is a continuous process of bounded variation on $[0,T]$,
\end{assumption}

\begin{theorem}
Let Assumptions~\ref{ass2} and~\ref{ass3} hold, $f\in {\mathbb C}^{1,1}\bigl({\mathbb R}^2\bigr)$. Then integral~\eqref{eqdfis} of $f(\mu_{t}, V_{t})$ with respect to $\mu_{t}$ is well defined, and
\begin{equation}
\int_{(0,T]}f(\mu_{t}, V_{t})\circ\,{\rm d}\mu_{t}
=F(\mu_{t},V_{t})-F(\mu_{0},V_{0})-\int_{(0,T]} F_2'(\mu_{t},V_{t})\,{\rm d}V_{t},\label{eqintf}
\end{equation}
where $F(x,v)=\int_{0}^x f(y,v)\,{\rm d}y$.
\end{theorem}

\begin{proof}
For $g\in {\mathbb C}^{2}\bigl({\mathbb R}\bigr)$ we have
\begin{eqnarray}
g(b)=g(a)+g'(a)(b-a)+\frac{1}{2}g''(a)(b-a)^2
+\int_a^b (g''(x)-g''(a))(b-x)\,{\rm d}x,\nonumber\\
\Bigl|\int_a^b (g''(x)-g''(a))(b-x)\,{\rm d}x\Bigr|\le \frac{1}{2}\max_{a\le x\le b}|g''(x)-g''(a)|(b-a)^2\label{eqestg}.
\end{eqnarray}

By similar way, for $h\in {\mathbb C}^{1}\bigl({\mathbb R}\bigr)$ we get
\begin{eqnarray}
h(b)=h(a)+h'(a)(b-a)+\int_a^b (h'(x)-h'(a))\,{\rm d}x,\nonumber\\
\Bigl|\int_a^b (h'(x)-h'(a))\,{\rm d}x\Bigr|\le \max_{a\le x\le b}|h'(x)-h'(a)|(b-a)\label{eqesth}.
\end{eqnarray}

So for $F\in {\mathbb C}^{2,1}\bigl({\mathbb R}^2\bigr)$ it follows
\begin{multline*}
F(x_1,v_1)-F(x,v)=F(x_1,v)-F(x,v)-(F(x_1,v)-F(x_1,v_1))\\
=F_1'(x,v)(x_1-x)+\frac{1}{2}F_{11}''(x,v)(x_1-x)^2+o(1)(x_1-x)^2
-F_2'(x_1,v_1)(v-v_1)+o(1)(v-v_1).
\end{multline*}
We considered there $o(1)$ as  $x_1\to x$, $v_1\to v$, and their values may be estimated by \eqref{eqestg} and \eqref{eqesth}.

If we change here $x_1\leftrightarrow x$,  $v_1\leftrightarrow v$ and take the difference of two equalities we get
\begin{multline}
F(x_1,v_1)-F(x,v)
=\frac{F_1'(x,v)+F_1'(x_1,v_1)}{2}(x_1-x)\\
+\frac{F_2'(x,v)+F_2'(x_1,v_1)}{2}(v_1-v)
+o(1)(x_1-x)^2+o(1)(v-v_1).\label{eqdiff}
\end{multline}
(Here we used that $F_{11}''(x_1,v_1)-F_{11}''(x,v)=o(1)$.)

Further, we denote $\Delta_{kn}=\bigl(t_{k-1}^n,t_{k}^n\bigr]$, and get
\begin{multline*}
F(\mu_{t},V_{t})-F(\mu_{0},V_{0})
=\sum_{k=1}^{j_n} \bigl(F(\mu_{t_{k}^n},V_{t_{k}^n})-F(\mu_{t_{k-1}^n},V_{t_{k-1}^n})\bigr)\\
=\sum_{k=1}^{j_n}  \frac{F_1'(\mu_{t_{k-1}^n},V_{t_{k-1}^n})
+F_1'(\mu_{t_{k}^n},V_{t_{k}^n})}{2} \mu(\Delta_{kn})
+\sum_{k=1}^{j_n} \frac{F_2'(\mu_{t_{k-1}^n},V_{t_{k-1}^n})+F_2'(\mu_{t_{k}^n},V_{t_{k}^n})}{2} \,\bigl(V_{t_{k}^n}-V_{t_{k-1}^n}\bigr)\\
 +o(1)\sum_{k=1}^{j_n} \mu^2(\Delta_{kn}) + o(1)\sum_{k=1}^{j_n} \bigl(V_{t_{k}^n}-V_{t_{k-1}^n}\bigr).
\end{multline*}
We take $n\to\infty$, use Corollary~\ref{crmdkn} and obtain~\eqref{eqintf}.

Note that for uniformly continuous functions $\mu_{t}$, $V_{t}$, $t\in [0, T]$, here we have $o(1)\to 0$ as $\max_k |t_k^n-t_{k-1}^n|\to 0$.
\end{proof}

\begin{theorem}
Let Assumptions~\ref{ass2} and~\ref{ass3} hold, $f\in {\mathbb C}^{1,1}\bigl({\mathbb R}^2\bigr)$, $g\in {\mathbb C}^{2,1}\bigl({\mathbb R}^2\bigr)$. Then the integral~\eqref{eqdfis} of $f(\mu_{t}, V_{t})$ with respect to $g(\mu_{t}, V_{t})$ is well de{f}{i}ned, and
\begin{multline}
\int_{(0,T]}f(\mu_{t}, V_{t})\circ\,{\rm d}g(\mu_{t}, V_{t})
=\int_{(0,T]} f(\mu_{t}, V_{t}) g_1'(\mu_{t},V_{t})\circ\,{\rm d}\mu_{t}
+\int_{(0,T]} f(\mu_{t}, V_{t}) g_2'(\mu_{t},V_{t})\,{\rm d}V_{t}.\label{eqinfv}
\end{multline}
\end{theorem}

\begin{proof} Denote $\zeta_{t}=f(\mu_{t}, V_{t})$, $u=t_{k}^n$, $s=t_{k-1}^n$, and consider
\begin{multline*}
\frac{\xi_{s}+\xi_{u}}{2}\,\bigl(g(\mu_{u}, V_{u})-g(\mu_{s}, V_{s})\bigr)
\stackrel{\eqref{eqdiff}}{=}\frac{\xi_{s}+\xi_{u}}{2}
\frac{g_1'(\mu_{s}, V_{s})+g_1'(\mu_{u}, V_{u})}{2}\mu(\Delta_{kn})\\
+\frac{\xi_{s}+\xi_{u}}{2} \frac{g_2'(\mu_{s}, V_{s})+g_2'(\mu_{u}, V_{u})}{2}(V_{u}-V_{s})
+o(1)\mu^2(\Delta_{kn})+o(1)(V_{u}-V_{s}).
\end{multline*}
For the first summand we have
\begin{multline*}
\frac{\xi_{s}+\xi_{u}}{2} \frac{g_1'(\mu_{s}, V_{s})+g_1'(\mu_{u}, V_{u})}{2}\mu(\Delta_{kn})
-\frac{\xi_{s}g_1'(\mu_{s}, V_{s})+\xi_{u}g_1'(\mu_{u}, V_{u})}{2}\mu(\Delta_{kn})\\
=\frac{\xi_{u}-\xi_{s}}{2}\frac{g_1'(\mu_{s}, V_{s})-g_1'(\mu_{u}, V_{u})}{2} \mu(\Delta_{kn})
\stackrel{\eqref{eqdiff}}{=}
\frac{\xi_{u}-\xi_{s}}{4}\Bigl(\frac{g_{11}'(\mu_{s}, V_{s})+g_{11}'(\mu_{u}, V_{u})}{2}\mu^2(\Delta_{kn})\\
+\frac{g_{12}'(\mu_{s}, V_{s})+g_{12}'(\mu_{u}, V_{u})}{2}\mu(\Delta_{kn})(V_{u}-V_{s})
+o(1)\mu^3(\Delta_{kn})+o(1)(V_{u}-V_{s})\mu(\Delta_{kn})\Bigr).
\end{multline*}
The sum over this terms for $1\le k\le j_n$ tends to 0 as $n\to \infty$ because $\xi_{u}-\xi_{s}=o(1)$, thus we obtain the integral with respect to $\mu$ in~\eqref{eqinfv}, and analogously
\[
\sum \frac{\xi_{s}+\xi_{u}}{2} \frac{g_2'(\mu_{s}, V_{s})+g_2'(\mu_{u}, V_{u})}{2}(V_{u}-V_{s})
\to \int_{(0,T]} f(\mu_{t}, V_{t}) g_2'(\mu_{t},V_{t})\,{\rm d}V_{t}.
\]
as usual Stieltjes integral.
\end{proof}

\section{SDE driven by SM}
\label{scsdes}

Let $\mu$ be a SM such that the process $\mu_{t}=\mu((0,t])$ is continuous.

In this section we study the stochastic equation of the form
\begin{equation}\label{eqstmu}
\circ\,{\rm d}X_{t}=\sigma(X_t)\circ\,{\rm d}\mu_t+b(X_t,t)\,{\rm d}t,\quad 0\le t\le T.
\end{equation}
We can guarantee the existence of the stochastic integral only for functions of the kind $f(\mu_{t}, V_{t})$. Therefore,
using a Doss--Sussmann transformation is a natural approach in this case because it gives solution as a function of the integrator. Our consideration will be similar to \cite[Section 4]{erru03}.

\begin{definition}\label{dfsleq2}  A process $X_{t}$, $0\le t\le T$ is a \emph{solution} to~\eqref{eqstmu} if:

1) $X_{t}=f(\mu_{t},Y_{t})$, $f\in{\mathbb C}^{2,1}({\mathbb R}^2)$, $Y$ is a continuous process of bounded variation;

2) for any process $Z_s=\psi(\mu_s,X_s)$, $\psi\in {\mathbb C}^{1,1}({\mathbb R}^2)$, we have
\begin{equation}\label{eqzetx}
\int_{(0,t]} Z_s\circ\,{\rm d}X_s=\int_{(0,t]} Z_s\sigma(X_s)\circ\,{\rm d}\mu_s+\int_{(0,t]}  Z_s b(X_s,s)\,{\rm d}s,\quad t\in [0,T].
\end{equation}
\end{definition}

\begin{assumption}\label{ass4}
1) $\sigma\in {\mathbb C}^2({\mathbb R})$ and the derivatives $\sigma'$, $\sigma''$ are bounded;

2) $b\in {\mathbb C}({\mathbb R}\times [0, T])$;

3) for each $C>0$ there exists a $L(C)$ such that
\[
|b(x,t)-b(y,t)|\le L(C)|x-y|,\quad |x|,\ |y|\le C;
\]

4) $|b(x,t)|\le K(1+|x|)$.

\end{assumption}

Let $F:{\mathbb R}^2\to{\mathbb R}$ be the solution of the equation
\begin{equation}\label{eqderf}
\frac{\partial F}{\partial  r}(r,x)=\sigma(F(r,x)),\quad F(0,x)=x,
\end{equation}
which exists globally because of our assumptions. Set $H(r,x)=F^{-1}(r,x)$, where the inverse is taken with respect to $x$. We have that $F,\ H\in {\mathbb C}^{2,2}({\mathbb R}^2)$ and
\begin{eqnarray}\label{eqprdh}
\frac{\partial H}{\partial  r}(r,x)=-\sigma(x)\frac{\partial H}{\partial  x} (r,x)
\end{eqnarray}
(see calculations in~\cite{rusval00}  (5.5)--(5.12)).

\begin{theorem} Let Assumptions~\ref{ass2} and~\ref{ass4} hold, $X_0$ be an arbitrary random variable. Then equation~\eqref{eqstmu} has a unique solution $X_t=F(\mu_t,Y_t)$, where $Y_t$ is the solution of the random equation
\begin{equation}\label{eqsolx}
Y_t=H(0,X_0)+\int_0^t \frac{\partial H}{\partial
x} (\mu_s, F(\mu_s,Y_s)) b(F(\mu_s,Y_s),s)\,{\rm d}s.
\end{equation}
\end{theorem}

\begin{proof} By \eqref{eqinfv}, for $X_s=F(\mu_s,Y_s)$, $Z_s=\psi(\mu_s,X_s)$,  we get
\begin{multline*}
\int_{(0,t]} Z_s\circ\,{\rm d}X_s = \int_{(0,t]} Z_s\frac{\partial F}{\partial  r}(\mu_s,Y_s)\circ\,{\rm d}\mu_s
 + \int_{(0,t]} Z_s \frac{\partial F}{\partial  x}(\mu_s,Y_s)\,{\rm d}Y_s\\
 \stackrel{\eqref{eqderf},\ \eqref{eqsolx}}{=}
\int_{(0,t]} Z_s \sigma\bigl(F(\mu_s,Y_s)\bigr)\circ\,{\rm d}\mu_s
 + \int_{(0,t]} Z_s \frac{\partial F}{\partial  x}(\mu_s,Y_s) \frac{\partial H}{\partial
x} (\mu_s, F(\mu_s,Y_s)) b(F(\mu_s,Y_s),s)\,{\rm d}s.
\end{multline*}
Taking into account that $\frac{\partial F}{\partial x}\cdot \frac{\partial H}{\partial  x}=1$, we arrive to~\eqref{eqzetx}.

Now, we will prove the uniqueness of the solution $X$. Take in~\eqref{eqzetx} $Z_t=\frac{\partial H}{\partial  x}(\mu_t,X_t)$ one obtains
\begin{multline}
\int_0^t \frac{\partial H}{\partial  x} (\mu_s, X_s)\circ\,{\rm d}X_s
\stackrel{\eqref{eqzetx}}{=}   \int_0^t \frac{\partial H}{\partial  x} (\mu_s,
X_s)\sigma(X_s)\circ\,{\rm d}\mu_s+\int_0^t  \frac{\partial H}{\partial  x} (\mu_s, X_s) b(X_s,s)\,{\rm d}s\\
\stackrel{\eqref{eqprdh}}{=}  -\int_0^t \frac{\partial H}{\partial  r} (\mu_s, X_s)\circ\,{\rm d}\mu_s+\int_0^t \frac{\partial
H}{\partial  x} (\mu_s, X_s) b(X_s,s)\,{\rm d}s\, .\label{eqitoy}
\end{multline}
Using~\eqref{eqinfv} for $Y_t=H(\mu_t,X_t)=H(\mu_t,f(\mu_t,Y_t))$ we get
\begin{multline*}
Y_t=Y_0+\int_{(0,t]} \Bigl(\frac{\partial H}{\partial  r} (\mu_s, X_s)+\frac{\partial H}{\partial  x} (\mu_s, X_s) f_1' (\mu_s, Y_s)\Bigr)\circ\,{\rm d}\mu_s
+\int_{(0,t]} \frac{\partial H}{\partial  x} (\mu_s, X_s)f_2'(\mu_s, Y_s)\,{\rm d}Y_s\\
=Y_0+\int_{(0,t]} \frac{\partial H}{\partial  r} (\mu_s, X_s)\circ\,{\rm d}\mu_s
+\int_{(0,t]} \frac{\partial H}{\partial  x} (\mu_s, X_s)\circ\,{\rm d}X_s
\stackrel{\eqref{eqitoy}}{=}Y_0+\int_0^t \frac{\partial
H}{\partial x} (\mu_s, X_s) b(X_s,s)\,{\rm d}s,
\end{multline*}
that coincides with~\eqref{eqsolx}.
\end{proof}

Note that existence and uniqueness of a solution to~\eqref{eqsolx} follows from the classical theory of differential equations. Detailed calculations may be found in Section~5.2~D of~\cite{karatz}.

\section{Some counterexamples}
\label{sccoun}

\begin{example}\label{exquad} (Continuous $\mu_{t}$ is not a finite quadratic variation process.)

We will show that finite quadratic variation approach  (see~\cite{rusval00}) can not be applied for SMs.

Let $\tilde{\varepsilon}_i$, $i\ge 1$, be independent Bernoulli random variables (${\sf P}\{\tilde{\varepsilon}_i=1\}={\sf P}\{\tilde{\varepsilon}_i=-1\}=1/2$), $T_1>2\pi$, and for a Borel subset ${\sf A}\subset [0,T_1]$ we put
\[
\mu({\sf A})=\sum_{i=1}^{\infty} \alpha_i{\tilde{\varepsilon}_i} \int_{\sf A}  \cos{i}t\,{\rm d}t.
\]
Here each $\alpha_i\in \{0,1\}$, but their exact values we will choose later.

For each Borel set ${\sf A}\subset[0,T_1]$ this series converges  a.s. because
\[
\sum_{i=1}^{\infty} \Bigl(\int_{\sf A}  \cos{i}t\,{\rm d}t\Bigr)^2<+\infty
\]
and $\mu$ is a SM by Theorem~\ref{thniko}.

Note that the paths of the process
\[
\mu_{s}=\mu((0,s])= \sum_{i=1}^{\infty}\alpha_i \tilde{\varepsilon}_i \frac{\sin is}{i},
\]
are continuous. Moreover,
\[
\Bigl(\sum_{2^j\le i< 2^{j+1}} \frac{\alpha_i}{i^2} \Bigr)^{1/2}=O(2^{-j/2}),
\]
and by Theorem 7.3 of~\cite{kahane} the paths of~$\mu_{s}$ are H\"{o}lder continuous with exponent $\gamma$ for any $\gamma<1/2$.

Now we consider the quadratic variation of $\mu_{s}$ on $[0,2\pi]$. By straightforward calculations we obtain
\[
\int_{(0,{2\pi}]}\frac{(\mu_{s+\varepsilon}-\mu_{s})^2}{\varepsilon}\,{\rm d}s=
4\pi \sum_{i=1}^{\infty} \alpha_i \frac{\sin^2({i}\varepsilon/2)}{i^2 \varepsilon}:= 4\pi f(\varepsilon).
\]

Note that for each finite set $D$
\begin{equation}\label{eqfind}
\sum_{i\in D} \alpha_i \frac{\sin^2({i}\varepsilon/2)}{i^2 \varepsilon}\to 0,\quad \varepsilon\to 0.
\end{equation}

We will use the equality
\[
\sum_{i=1}^{\infty} \frac{\sin^2({i}\varepsilon/2)}{i^2 \varepsilon}=\frac{2\pi-\varepsilon}{8},
\]
that follows from Parseval's identity for the function $f(x)={\mathbf 1}_{[-\varepsilon/2,\varepsilon/2]}(x)$ on $[-\pi,\pi]$.

We will define intervals of positive integers $[m_j,n_j]$. Set $\alpha_i=1$ if $i\in\cup_{j\ge 1} [m_j,n_j]$ and $\alpha_i=0$ else.

Put $m_1=1$, $\varepsilon_1=1$ and take $n_1$ such that
\[
\sum_{i\in [m_1, n_1]} \frac{\sin^2({i}\varepsilon_1/2)}{i^2 \varepsilon_1}>\frac{1}{2}.
\]
Put $D_1=[m_1, n_1]$, so $f(\varepsilon_1)>1/2$.

Using \eqref{eqfind} we take $\varepsilon_2<\varepsilon_1$ so that
\[
\sum_{i\in [m_1, n_1]} \frac{\sin^2({i}\varepsilon_2/2)}{i^2 \varepsilon_2}<\frac{1}{8} .
\]

We take $m_2$ so that
\[
\sum_{i\ge m_2} \frac{\sin^2({i}\varepsilon_2/2)}{i^2 \varepsilon_2}<\frac{1}{8} ,
\]
so $f(\varepsilon_2)<1/4$.

Take $\varepsilon_3<\varepsilon_2$ such that
\[
\sum_{i< m_2} \frac{\sin^2({i}\varepsilon_3/2)}{i^2 \varepsilon_3}<\frac{2\pi-1}{8}-\frac{1}{2} ,
\]
then
\[
\sum_{i\ge m_2} \frac{\sin^2({i}\varepsilon_3/2)}{i^2 \varepsilon_3}>\frac{1}{2}.
\]
For given $\varepsilon_3$ we choose $n_2>m_2$ so that
\[
\sum_{i\in [m_2, n_2]} \frac{\sin^2({i}\varepsilon_3/2)}{i^2 \varepsilon_3}>\frac{1}{2}.
\]

We repeat this procedure, and obtain $f(\varepsilon_{2k-1})>1/2$, $f(\varepsilon_{2k})<1/4$. Obviously, we can choose $\varepsilon_n\downarrow 0$. Therefore, for $f$  constructed in this way the limit $\lim_{\varepsilon\to 0} f(\varepsilon)$ does not exist.
\end{example}

\begin{example} ($\lim_{n\to\infty}\sum_{k=1}^{j_n} \mu^2(\Delta_{kn})$ does not exist.)

This example will demonstrate that non-anticipated integral as
\[
{\rm p}\lim_{n\to\infty}\sum_{k=1}^{j_n}f(\mu_{t_{k-1}^n})\,\mu\bigl(\Delta_{kn}\bigr)
\]
can not be properly defined for SMs.

As in the previous example, we will construct a SM of the kind
\[
\mu({\sf A})=\sum_{i=1}^{\infty} \alpha_i{\tilde{\varepsilon}_i} \int_{\sf A}  \cos{i}t\,{\rm d}t,\quad {\sf A}\subset [0,2\pi],
\]
where $\tilde{\varepsilon}_i$, $i\ge 1$, are independent Bernoulli random variables, and $\alpha_i\in \{0,1\}$.

We consider
\[
\Delta_{kn}= \bigl(2^{-n}(k-1)\cdot 2\pi, 2^{-n}k\cdot 2\pi\bigr],\quad 1\le k\le 2^n,\quad n\ge 1.
\]

So we infer
\begin{multline*}
S_n:=\sum_{k=1}^{2^n} \mu^2(\Delta_{kn})= \sum_{k=1}^{2^n} \Bigl(\sum_{i=1}^{\infty} \alpha_i\tilde{\varepsilon}_i\int_{\Delta_{kn}}  \cos it\,{\rm d}t\Bigr)^2
=\sum_{k=1}^{2^n} \Bigl(\sum_{i=1}^{\infty} \alpha_i\tilde{\varepsilon}_i \frac{2\sin i 2^{-n}\pi\cdot\cos i 2^{-n}(2k-1)\pi}{i}\Bigr)^2\\
= 4\sum_{k=1}^{2^n} \sum_{1\le i,j<\infty} \alpha_i\tilde{\varepsilon}_i
 \frac{\sin i 2^{-n}\pi\cdot\cos i 2^{-n}(2k-1)\pi}{i}\alpha_j\tilde{\varepsilon}_j
 \frac{\sin j 2^{-n}\pi\cdot\cos j 2^{-n}(2k-1)\pi}{j}\\
 = 2\sum_{1\le i,j<\infty} \frac{\alpha_i\tilde{\varepsilon}_i\alpha_j\tilde{\varepsilon}_j}{ij} \sin i 2^{-n}\pi\cdot \sin j 2^{-n}\pi\sum_{k=1}^{2^n} \bigl(\cos (i-j) 2^{-n}(2k-1)\pi +\cos (i+j) 2^{-n}(2k-1)\pi\bigr).
\end{multline*}

We have
\begin{eqnarray*}
\sum_{k=1}^{2^n} \cos (i-j) 2^{-n}(2k-1)\pi= 0 \quad {\rm for}\quad (i-j)2^{-n}\not\in{\mathbb Z},\\
\sum_{k=1}^{2^n} \cos (i+j) 2^{-n}(2k-1)\pi= 0 \quad {\rm for}\quad (i+j)2^{-n}\not\in{\mathbb Z}.
\end{eqnarray*}

If $\alpha_i=1$ for $2^{n-2}\le i< 2^{n-1}$ and $\alpha_i=0$ for $i\ge 2^{n-1}$, we get
\begin{equation}\label{eqmukn}
S_n \ge 2\sum_{2^{n-2}\le i< 2^{n-1}} \frac{{2^n}}{i^2} \sin^2 (i 2^{-n}\pi)
\ge 8 \sum_{2^{n-2}\le i< 2^{n-1}} 2^{-n}= 2
\end{equation}
(here we have used the estimate $\sin x\ge x(2/\pi)$, $0\le x\le \pi/2$).

We will define an increasing sequence of positive integers $n_j$, set $\alpha_i=1$ if $i\in\cup_{j\ge 1} [2^{n_j-2},2^{n_j-1}-1]$ and $\alpha_i=0$ else.

In this case for $n>n_j$ we get
\begin{eqnarray*}
S_n  = \sum_{k=1}^{2^n} \Bigl(\sum_{1\le i<2^{n_j-1}} \alpha_i\tilde{\varepsilon}_i \int_{\Delta_{kn}}  \cos {i}t\, {\rm d}t+\sum_{i\ge 2^{n_{j+1}}} \alpha_i\tilde{\varepsilon}_i \int_{\Delta_{kn}}  \cos {i}t\,{\rm d}t\Bigr)^2\\
\le 2\sum_{k=1}^{2^n} \Bigl(\sum_{1\le i<2^{n_j-1}} \alpha_i\tilde{\varepsilon}_i \int_{\Delta_{kn}}  \cos {i}t\, {\rm d}t\Bigr)^2 +2\sum_{k=1}^{2^n} \Bigl(\sum_{i\ge 2^{n_{j+1}}} \alpha_i\tilde{\varepsilon}_i \int_{\Delta_{kn}}  \cos {i}t\, {\rm d}t\Bigr)^2
=:2A+2B
\end{eqnarray*}
As in~\eqref{eqmukn}, for fixed $n_j$ we have
\[
A=2\sum_{1\le i<2^{n_j-1}} \frac{\alpha_i\cdot {2^n}}{i^2} \sin^2 (i 2^{-n}\pi)\to 0,\quad n\to\infty,
\]
and $A<1/4$ for $n:=\tilde{n}_j$ large enough.

For fixed $\tilde{n}_j$ we can find large $n_{j+1}$ so that $B<1/4$, then $S_{\tilde{n}_j}<1$.
From other side, $S_{{n}_j}\ge 2$. Therefore, $\lim_{n\to\infty} S_n$ does not exist.

Note that by Theorem 7.3 of~\cite{kahane} the paths of~$\mu_{s}$ are H\"{o}lder continuous with exponent $\gamma$ for any $\gamma<1/2$ (as in Example~\ref{exquad}).

\end{example}

\section*{Acknowledgments}

The author acknowledges the support of Alexander von Humboldt Foundation under grant 1074615, and thanks
Prof. M.~Z\"{a}hle for fruitful discussions during the preparation of this paper.


\begin{thebibliography}{99}
\bibitem{covrus}
R. Coviello and F. Russo, \emph{Nonsemimartingales: stochastic
  differential equations and weak Dirichlet processes}, Ann.
  Probab. 35 (2007), pp. 255--308.

\bibitem{curbera}
G. Curbera and O. Delgado, \emph{Optimal domains for ${L}^0$-valued operators
  via stochastic measures}, Positivity 11 (2007), pp. 399--416.

\bibitem{deya13}
A. Deya, M. Jolis, and L. Quer-Sardanyons, \emph{The Stratonovich heat
  equation: a continuity result and weak approximations}, Electron. J. Probab
  18 (2013), pp. 1--34.

\bibitem{dretop}
L. Drewnowski, \emph{Topological rings of sets, continuous set functions,
  integration.~{III}}, Bull. Acad. Pol. Sci. S\'{e}r. sci. math. astron. phys.
  20 (1972), pp. 439--445.

\bibitem{erru03}
M. Errami and F. Russo, \emph{n-covariation, generalized {D}irichlet processes
  and calculus with respect to {f}{i}nite cubic variation processes}, Stoch.
  Proc. Appl. 104 (2003), pp. 259 -- 299.

\bibitem{hu2013}
Y. Hu, M. Jolis, and S. Tindel,  \emph{On {S}tratonovich and
  {S}korohod stochastic calculus for Gaussian processes}, Ann.
  Probab. 41 (2013), pp. 1656--1693.

\bibitem{kahane}
J.P. Kahane, \emph{Some Random Series of Functions}, Cambridge University Press, 1993.

\bibitem{karatz}
I. Karatzas and S. Shreve, \emph{Brownian Motion and Stochastic Calculus}, Springer Science \& Business Media, 2012.

\bibitem{kurtz}
T.G. Kurtz, {\'E}. Pardoux, and P. Protter, \emph{Stratonovich stochastic
  differential equations driven by general semimartingales}, Ann.
  I. H. Poincare Prob. 31 (1995), pp. 351--377.

\bibitem{kwawoy}
S. Kwapie\'{n} and W.A. Woyczy\'{n}ski, \emph{Random Series and Stochastic Integrals: Single and Multiple}, Birkh\"{a}user, Boston, 1992.

\bibitem{memiva}
T. Memin, Y. Mishura, and E. Valkeila, \emph{Inequalities for the moments of
  {W}iener integrals with respect to a fractional {B}rownian motion}, Statist.
  Probab. Lett. 51 (2001), pp. 197--206.

\bibitem{radmon}
V. Radchenko, \emph{Integrals with respect to General Stochastic Measures}, Institute of Mathematics, Kyiv 1999, in Russian.

\bibitem{rads06}
V. Radchenko, \emph{Besov regularity of stochastic measures}, Statist. Prob.
  Lett. 77 (2007), pp. 822--825.

\bibitem{rads09}
V. Radchenko, \emph{Mild solution of the heat equation with a general
  stochastic measure}, Stud. Math. 194 (2009), pp. 231--251.

\bibitem{radt09}
V. Radchenko, \emph{Paths of stochastic measures and {B}esov spaces}, Theory
  Probab. Appl. 54 (2010), pp. 160--168.

\bibitem{radc14}
V. Radchenko, \emph{Stochastic partial differential equations driven by general
  stochastic measures}, in \emph{Modern Stochastics and Applications},
  Springer,  2014, pp. 143--156.

\bibitem{rusval00}
F. Russo and P. Vallois, \emph{Stochastic calculus with respect to continuous
  finite quadratic variation processes}, Stoch. Stoch. Rep. 70
  (2000), pp. 1--40.

\bibitem{samtaq}
G. Samorodnitsky and M. Taqqu, \emph{Stable Non-Gaussian Random Processes}, Chapman \& Hall, London,  1994.

\bibitem{swanson}
J. Swanson, \emph{The calculus of differentials for the weak {S}tratonovich
  integral}, in \emph{Malliavin Calculus and Stochastic Analysis}, Springer,
  2013, pp. 95--111.

\bibitem{vahtar}
N.N. Vakhania, V.I. Tarieladze, and S.A. Chobanian, \emph{Probability Distributions on Banach Spaces}, D.~Reidel Publishing Co.,  Dordrecht, 1987.
\end{thebibliography}
\end{document}